%% file: main.tex
\documentclass[a4paper]{amsart}
\usepackage{graphicx}
\usepackage{hyperref}
\usepackage{pdfsync}

\usepackage[T1]{fontenc}
\usepackage[utf8]{inputenc} 

\usepackage{amsmath, amssymb, amsfonts, amscd, amsthm}

\newcommand{\maslov}[3]{\tau( #1,~#2,~#3    )}
\newcommand{\m}[3]{\maslov{\Gamma_{#1}}  {\Gamma_{#2}} {\Gamma_{#3}} }

\newcommand{\B}{B}

\newcommand{\Tr}{\operatorname{Tr}\,}
\newcommand{\Sp}{\mathbf{Sp}}
\newcommand{\Mp}{\mathbf{Mp}}
\newcommand{\U}{\mathbf{U}}
\newcommand{\WH}{W\!H}
\newcommand{\D}[1]{\mathcal{D}(#1)}

\newcommand{\Z}{\mathbb{Z}}
\newcommand{\N}{\mathbb{N}}
\newcommand{\C}{\mathbb{C}}
\newcommand{\R}{\mathbb{R}}
\newcommand{\Q}{\mathbb{Q}}
\newcommand{\F}{\mathbb{F}}

\newcommand{\z}{\mathbb{Z}}
\newcommand{\q}{\mathbb{Q}}

\newtheoremstyle{pedro}{}{}{\itshape}{}{\sc}{~--}{ }{\thmname{#1}\thmnumber{ #2}\thmnote{ (#3)}}

\newtheoremstyle{pedrodef}{}{}{}{}{\sc}{~--}{ }{\thmname{#1}\thmnumber{ #2}\thmnote{ (#3)}}

\theoremstyle{pedro}
\newtheorem{lem}{Lemma}[section]

\newtheorem{thm}[lem]{Theorem}

\newtheorem{prop}[lem]{Proposition}

\theoremstyle{remark}

\newtheorem{rmk}[lem]{Remark}

\theoremstyle{pedrodef}

\title{A link invariant with values in the Witt ring}
\author{Ga\"el Collinet and Pierre Guillot}

\address{
Universit\'{e} de Strasbourg \& CNRS\\
Institut de Recherche Math\'{e}matique Avanc\'{e}e\\
7~Rue Ren\'{e} Descartes\\
67084 Strasbourg, France}

\email{collinet@math.unistra.fr, guillot@math.unistra.fr}

\numberwithin{equation}{section}

\begin{document}

\maketitle

\begin{abstract}
  Using Maslov indices, we show the existence of oriented link
  invariants with values in the Witt rings of certain fields. Various
  classical invariants are closely related to this construction. We
  also explore a surprising connection with the Weil representation.
\end{abstract}


\input{newintro}
\input{def-maslov}

\input{facile}

\input{examples}

\input{weil}

\bibliography{myrefs}
\bibliographystyle{alpha}

\end{document}

%% file: newintro.tex
\section{Introduction}

In this paper we show that, given an appropriate field~$K$, one can
associate to any oriented link~$L$ in Euclidean~$3$-space an
element~$\Theta_K(L) \in W(K)$, where~$W(K)$ is the {\em Witt ring}
of~$K$. This element~$\Theta_K (L)$ is an isotopy invariant. 

Recall that an element in~$W(K)$ is given by a (non-degenerate)
quadratic form~$q$, over a finite dimensional~$K$-vector
space. However, two such quadratic forms~$q$ and~$q'$ may define the
same element in~$W(K)$ even if they are not isomorphic; in this case
we call them Witt equivalent. Witt equivalence can be more or less
subtle, depending strongly on the field~$K$. The easiest example is
that of~$K=\R$, the field of real numbers, for in this case~$q$
and~$q'$ are Witt equivalent precisely when they have the same
signature. Accordingly one has an isomorphism~$W(\R) \cong \Z$. By
contrast, $W(\Q)$ is considerably more complicated (see below).

There have been many efforts to use quadratic forms in order to define
link invariants, and the pattern has frequently been as follows: one
has a procedure to obtain a quadratic form from a link, but it is not
itself an isotopy invariant, so that one is reduced to extracting
coarser information. As early as 1932, Reidemeister in~\cite[\S7 and
\S8]{reide} uses the so-called ``Minkowski units'' of a certain
quadratic form of his design, and proves that they are
invariant. Better known is the construction of the {\em signature} of
a link, which is really the signature of a certain non-invariant
quadratic form (see~\cite{lick}, Theorem 8.9 with~$\omega=-1$).  In
retrospect it may be said, rather pedantically, that the quadratic
form is replaced by its Witt equivalence class in~$W(\R)$ in order to
get an invariant. In a sense, {\em in this paper we use the same
  strategy of working with the Witt ring, but over other
  fields}. However our construction is not as direct, and breaks the
above pattern.

Indeed we shall follow in the footsteps of Ghys and Gambaudo
(see~\cite{ghys}), who were explicitly thinking of the signature as
taking its values in~$W(\R)$ in order to solve the following
problem. Let~$s(L)$ denote the signature of the oriented link~$L$. If
we consider the {\em braid group}~$B_n$ on~$n$ strands, then we may
define a map~$f_n \colon B_n \to \Z$ by~$f_n(\beta ) = s( \hat \beta
)$. Here we use~$\hat \beta $ to denote the {\em closure} of the
braid~$\beta $, which is an oriented link in~$\R^3$. One may ask
whether~$f_n$ is a group homomorphism; it is not, and indeed Ghys and
Gambaudo obtain an explicit formula for
\[ c(\beta , \gamma ) :=   f_n(\beta \gamma ) - f_n(\beta ) - f_n(\gamma )\,
.   \]
Their formula is in terms of quadratic forms (as opposed to plain
integers). In spite of the complicated notation, let us give it here:
\[  f_n(\beta \gamma ) - f_n(\beta ) - f_n(\gamma )= \maslov{\Gamma_1}{\Gamma_{r(\beta
    )}}{\Gamma_{r(\beta \gamma )}} \, . \tag{*}
  \]
  Here~$r \colon \B_n \to GL_n(\R)$ is the {\em Burau representation}
  (``at~$t=-1$''), the notation~$\Gamma_g$ is used for the graph
  of~$g$, and~$\tau$ is the {\em Maslov index}, an algebraic
  construction which produces quadratic forms (up to Witt
  equivalence). 

  What we do in this paper is {\em to take (*) as a definition of a
    link invariant} instead. More precisely, we construct for each~$n$
  a map~$f_n \colon B_n \to W(K)$, for a suitable field~$K$, such that
  the analog of (*) holds.  This makes sense since the Maslov index is
  a very general procedure, not constrained to~$K= \R$. Then, we show
  that~$(f_n)_{n \ge 2}$ is a {\em Markov function}, that is, it is
  compatible with the Markov moves. The celebrated theorems of
  Alexander and Markov then imply that~$f_n(\beta ) = \Theta_K( \hat
  \beta )$ for some link invariant~$\Theta_K$. In particular, we have
  the following result.

\begin{thm} \label{thm-example}
Let~$K= \R$, or~$\Q$, or a finite field, or~$\Q(t)$, the field of
rational fractions in~$t$. Then there exists a unique oriented link
invariant~$\Theta_K$ with values in the Witt ring~$W(K)$, which takes
the zero value for disjoint unions of unknots, and with the following
extra property. Defining~$f_n \colon B_n \to W(K)$ by~$f_n(\beta ) =
\Theta_K(\hat \beta )$, one has 
\[   f_n(\beta \gamma )- f_n(\beta ) - f_n(\gamma ) =
\maslov{\Gamma_1}{\Gamma_{r_n(\beta )}}{\Gamma_{r_n(\beta \gamma )}} \,
. \tag{**} \]
\end{thm}

Here~$r_n \colon B_n \to GL_n(K)$ is the appropriate version of the
Burau representation. (We caution the reader who may glance at the
results in the text now that for~$K= \Q(t)$ we actually mention a link
invariant with values in a ring written~$\WH(\Q(t))$ and called the
hermitian Witt ring; luckily~$\WH(\Q(t)) \subset W(\Q(t))$ in this
case and the theorem holds as stated. These details need not distract
us now.) This Theorem appears in the text as Theorem~\ref{thm-main}.

When it comes to computing~$\Theta_K(L)$ explicitly, we have to rely
on (**), after having found a braid group element whose closure
is~$L$. We hasten to add that we have made a {\sc Sage} script
available, which can perform the calculations automatically. It
outputs a diagonal matrix representing the quadratic
form~$\Theta_K(L)$. In the rest of this Introduction, we assume that
the computational side of things is thus taken care of, and comment on
the results. Before anything else though, this is as good a place as
any to point out that (**) implies that~$f_n(\beta ) = -f_n(
\beta^{-1})$, so that~$\Theta_K(L') = - \Theta_K(L)$ if~$L'$ is the
mirror-image of~$L$. In particular~$2\Theta_K(L) = 0$ if~$L$ and~$L'$
are isotopic.

First and foremost, for~$K= \R$ one can interpret the result by Ghys
and Gambaudo as saying that~$\Theta_\R(L)$ agrees with the signature
of~$L$. Things are already more interesting with~$K= \Q$. In this case
(see~\cite{milnor}) there is an exact sequence (where the arrows are
completely explicit)
$$ 0 \longrightarrow \Z \longrightarrow W(\Q) \longrightarrow
\bigoplus_{p} W(\F_p) \longrightarrow 0 \, .   $$
This sequence is split by the homomorphism~$W(\Q) \to W(\R) \cong
\Z$. Moreover, for~$p$ odd the group $W(\F_p)$ is either~$\Z/2\times
\Z/2$ or~$\Z/4$ according as~$p$ is~$1$ mod~$4$ or not, while~$W(\F_2)
= \Z/2$.  For each oriented link~$L$, we obtain a set of primes which
is an invariant of~$L$, namely the set of those~$p$ for
which~$\Theta_\Q(L)$ maps to a non-zero element via the residue
map~$W(\Q) \to W(\F_p)$. Of course, for each~$p$ the value
in~$W(\F_p)$ is also an invariant.

The truly interesting case is~$K=\Q(t)$. The ring~$W(\Q(t))$ is very
rich, so the first thing we should do is extract easily computable
information from~$\Theta_{\Q(t)}(L)$.  We do this by showing that the
above theorem yields a Laurent-polynomial invariant akin to the
Alexander-Conway polynomial. We also exploit our method to produce a
``signature'' for each complex number~$\omega $ of module~$1$, that is
a~$\Z$-valued invariant. When~$\omega$ is a root of unity, this
invariant is related to the Levine-Tristram signature, as follows
again from~\cite{ghys}. Thus~$\Theta_{\Q(t)}(L)$ seems to ``contain''
many other invariants, and its first virtue is unification.

However, there is more to~$\Theta_{\Q(t)}(L)$ than the polynomial and
the signatures. One has an exact sequence
$$ 0 \longrightarrow W(\Q) \longrightarrow W(\Q(t))
\longrightarrow \bigoplus_P  W(\Q[t]/(P) )\longrightarrow 0 \, .   $$
Here the direct sum runs over all irreducible polynomials, so
that~$\kappa= \Q[t]/(P)$ is a number field. Finally, the Witt
ring~$W(\kappa)$ fits into yet another exact sequence, similar to that
for~$\Q$ but involving the Witt ring of the ring of integers
in~$\kappa$ (for~$\kappa= \Q$, this is the ring~$\Z$, and~$W(\Z) =
\Z$; we shall not encounter Witt rings of rings which are not fields
elsewhere in this paper). All the arrows are quite explicit, so even
though~$W(\Q(t))$ appears to be huge, it is in principle always
possible to decide in finite time whether~$\Theta_{\Q(t)}(L)$ is zero
(and thus possibly show that~$L$ is not the trivial knot). 

So far we have described the contents of the three sections of the
paper following this Introduction. In Section~\ref{sec-def-maslov} we
present background material and give simple, sufficient conditions for
an invariant as above to be defined out of representations of the
braid groups. In Section~\ref{sec-burau} it is shown that these
conditions are satisfied in the case of the Burau
representation. Examples are provided in Section~\ref{sec-examples}.

Let us now say a word about Section~\ref{sec-weil}, which explores the
ideas behind the proof of Theorem~\ref{thm-example} in the
case~$K=\R$, rather than its statement, and connects them to the
so-called Weil representation. In summarizing Section~\ref{sec-weil}
we shall presently provide a sketch of the key steps in the proof of
the Theorem (the assumption~$K=\R$ allowing for simpler arguments).

The first ingredient is the observation that the Burau representation
at~$t=-1$ carries a~$B_{2n}$-invariant symplectic form, thus providing
a map~$r_{2n} \colon B_{2n} \to \Sp_{2n}(\R)$. Now, we
have~$\pi_1(\Sp_{2n}(\R)) = \Z$, so that~$\Sp_{2n}(\R)$ possesses many
covers; we shall be particularly interested in the simply-connected
cover~$\widetilde{\Sp}_{2n}(\R)$ and the~$2$-fold cover~$M_{2n}$, also
known as the metaplectic group.

The second ingredient is the cohomological fact that~$H^2(B_n, \Z) =
0$, which implies that the map~$r_{2n} \colon B_{2n} \to
{\Sp}_{2n}(\R)$ can be lifted to a map~$r'_{2n} \colon B_{2n} \to
\widetilde{\Sp}_{2n}(\R)$. The kernel of the
map~$\widetilde{\Sp}_{2n}(\R) \to {\Sp}_{2n}(\R)$ is~$\Z$, and once we
describe~$\widetilde{\Sp}_{2n}(\R)$ explicitly using a two-cocycle~$c$
with values in~$\z$, then finding~$r'_{2n}$ amounts to finding a
one-cocycle on~$B_{2n}$ whose coboundary is~$c$. That one-cocycle
is the map~$f_{2n}$ which appears in Theorem~\ref{thm-example}.

These two ingredients must be slightly refined in the case of a
general field~$K$, but the spirit of the construction of the map~$f_n$
is always the same. Topological arguments are replaced by the
apparatus of Maslov indices (which are needed in order to make precise
statements anyway).

In Section~\ref{sec-weil}, the emphasis is on the induced map~$B_{2n}
\to M_{2n}$, for~$M_{2n}$ is known to act on an infinite-dimensional Hilbert
space via the Weil representation. Thus~$B_{2n}$ also acts on this
space, and the representation is also known to have a ``trace'' in
some technical sense. This trace has been computed by Thomas
(\cite{thomas}), who provides explicit formulae involving Maslov
indices. Comparing these with the material in
Section~\ref{sec-def-maslov}, we end up proving that the trace is a
link invariant, which can be expressed in terms of the
Alexander-Conway polynomial and the signature (which appears in the
guise of~$\Theta_\R$, see above). Note that strictly speaking our
result about the trace of the Weil representation is not deduced from
Theorem~\ref{thm-example}; rather, it consitutes a variant on its
proof. 

We conclude the paper with some remarks about the Weil representation
of finite fields and the work of Goldschmidt and Jones.

In a subsequent paper it will be established that, at the price of
more machinery including a recent theorem of Barge and Lannes on
Maslov indices over rings, we can follow the above outline
over~$\z[\frac{1} {2}, t, t^{-1} ]$ instead of a field. There results
a single link invariant which specializes to all the others, thus
pushing the unification a step further. What is more, it will be shown
that our method extends to the case of {\em coloured links}, for which
the braid groups have to be replaced by an appropriate groupoid.

\noindent {\em Acknowledgments.} The authors wish to thank \'Etienne
Ghys, Jean Barge, and Christian Kassel for their interest in the
paper. Our thanks extend to Hubert Rubenthaler for helpful discussions
on the harmonic analysis underlying the Weil representation. Pierre
Torasso pointed out the reference~\cite{thomas}, and we are grateful
for his help. Finally, we are indebted to Ivan Marin for discovering
an embarrassing mistake in an earlier version of the paper. Also, we
would like to thank the referee for raising subtle technical points
about \S\ref{sec-weil} and generally encouraging us to develop that
section. 

%% file: def-maslov.tex
\section{Background material} \label{sec-def-maslov}

\subsection{The braid groups}
The {\em braid group on~$n$ strands}~$\B_n$ is the group generated
by~$n-1$ generators~$\sigma_1, \ldots, \sigma_{n-1}$ subject to the
relations~$\sigma_i \sigma_j = \sigma_j \sigma_i$ for $|i - j|>2$, while 
$$ \sigma_i \sigma_{i+1} \sigma_i = \sigma_{i+1} \sigma_i
\sigma_{i+1} \, .   $$
The well-known interpretation of~$\B_n$ in terms of geometric braids
(see \cite{kassel}, Theorem~1.12) allows one to define the operation
of {\em closure} $\beta \mapsto \hat \beta $ ({\em loc.\ cit.},
\S2.2): here~$\beta \in\B_n$ and~$\hat \beta $ is an oriented link in
Euclidean~$3$-space. The celebrated theorem of Alexander ({\em loc.\
  cit.}, \S2.3) asserts that any oriented link in~$\R^3$ is isotopic
to one of the form~$\hat \beta $ for some~$\beta $ belonging to
some~$\B_n$.

This process defines an equivalence relation on the disjoint
union~$\coprod_{n\ge 2} \B_n$, according to which~$\beta \sim \gamma $
whenever the links~$\hat \beta $ and~$\hat \gamma $ are
isotopic. Markov's theorem ({\em loc.\ cit.}, \S2.5) describes this
relation explicitly. Here we shall state the result in the following
form: a map~$f= \coprod_{n\ge 2} f_n$ on the above disjoint union,
with values in any set~$E$, is constant on the equivalence classes if
and only if the two following properties are satisfied:

(i) each $f_n$ is invariant under conjugation, that
is~$f_n(\gamma^{-1} \beta \gamma ) = f_n (\beta )$ for all~$\beta,
\gamma \in \B_n$.

(ii) for all~$n\ge 2$ and all~$\beta \in \B_n$ one has~$f_{n+1}(\iota_n(\beta)
\sigma_n^{\pm 1}) = f_n(\beta )$, where~$\iota_n$ denotes the inclusion
of~$\B_n$ into~$\B_{n+1}$.

Such a map is usually called a {\em Markov function}. It follows that
the value of a Markov function on a braid~$\beta $ only depends on the
closure~$\hat \beta $, and in view of Alexander's theorem we see that
a Markov function gives an oriented link invariant (and conversely).

Because of condition (i), a traditional strategy in order to produce
Markov functions is to start with a collection of
representations~$r_n : \B_n \to GL(V_n)$, where~$V_n$ is a module
over some ring~$R$, and then rely on functions which are known to be
conjugation-invariant on the group of invertible matrices, like the
trace or determinant. A standard example is the Alexander-Conway
polynomial, which relies on the Burau representation and the
determinant ({\em loc.\ cit.}, \S3.4). Here~$R= \z[t, t^{-1}]$ and the
invariant takes its values in~$R$.

\subsection{Witt rings and Maslov indices}
\label{subsec-maslov}
Let~$K$ be a field of characteristic different from~$2$. Suppose that~$K$ is
endowed with an involution~$\sigma $, and let~$k=K^\sigma$ denote the
field of fixed elements. We shall write~$\bar x$ instead of~$\sigma
(x)$.

Let~$V$ be a vector space over~$K$. A map~$h : V\times V \to K$ is
called an {\em anti-hermitian form} (resp.\ a hermitian form) when it
is linear in one variable and satisfies
$$  h(y, x) = - \overline{h(x, y)} \qquad (\textnormal{resp.}~h(y, x) =
\overline{h(x, y)}) \, .    $$
In this case~$V$ is called an anti-hermitian space (resp.\ a hermitian
space). The form~$h$ is called non-degenerate when the determinant of
the corresponding matrix (in any basis) is non-zero.

Let~$V$ be anti-hermitian. A {\em lagrangian} is a
subspace~$\ell\subset V$ such that~$\ell = \ell^\perp$. We say
that~$V$ is {\em hyperbolic} when it is the direct sum of two
lagrangians.

Now given a hyperbolic, non-degenerate, anti-hermitian space~$V$ with
form~$h$ and three lagrangians~$\ell_1, \ell_2$ and~$\ell_3$, we shall
describe their {\em Maslov index}, which is a certain element
$$ \maslov {\ell_1} {\ell_2} {\ell_3} \in \WH(K, \sigma ) \, .   $$
Here~$\WH(K, \sigma )$ is the hermitian Witt ring of~$K$: see
\cite{milnor}. For example~$\WH(K, \sigma )$ may be defined as the
quotient of the Grothendieck ring of the category of non-degenerate
hermitian spaces by the ideal consisting of all hyperbolic
spaces. 

The Maslov index~$\maslov {\ell_1} {\ell_2} {\ell_3} $ is
then the non-degenerate space corresponding to the following hermitian
form on~$\ell_1\oplus\ell_2\oplus\ell_3$:
$$ H( \mathbf{v}, \mathbf{w}) = h(v_1, w_2 - w_3) + h(v_2, w_3 - w_1)
+ h(v_3, w_1 - w_2) \, .   $$
More precisely, if this hermitian space is degenerate, we take the
quotient by its kernel.

We claim that this construction enjoys the following properties:

(i) Dihedral symmetry:
$$ \maslov {\ell_1} {\ell_2} {\ell_3} = - \maslov {\ell_3} {\ell_2}
{\ell_1} = \maslov{\ell_3}{\ell_1}{\ell_2} \, .   $$

(ii) Cocycle condition:
$$ \maslov {\ell_1} {\ell_2} {\ell_3} + \maslov {\ell_1} {\ell_3}
{\ell_4} = \maslov {\ell_1} {\ell_2} {\ell_4} + \maslov {\ell_2}
{\ell_3} {\ell_4} \, .      $$

(iii) Additivity: if~$\ell_1, \ell_2$ and~$\ell_3$ are lagrangians
in~$V$, while~$\ell'_1, \ell'_2$ and~$\ell'_3$ are lagrangians
in~$V'$, then~$\ell_i\oplus\ell'_i$ is a lagrangian in the orthogonal
direct sum~$V\oplus V'$ and we have
$$ \maslov {\ell_1\oplus\ell'_1} {\ell_2\oplus\ell'_2} {\ell_3\oplus\ell'_3} = \maslov {\ell_1} {\ell_2}
{\ell_3}  + \maslov {\ell'_1} {\ell'_2} {\ell'_3}  \, .  $$

(iv) Invariance: for any~$g\in \U(V)$ (the unitary group), one has
$$ \maslov{g\cdot\ell_1}{g\cdot\ell_2}{g\cdot\ell_3} = \maslov{\ell_1}{\ell_2}{\ell_3} \, .   $$

In fact, in the particular case when~$\sigma = Id$, and thus~$k=K$, an
anti-hermitian form is nothing but a symplectic form, and a hermitian
form is just a symmetric, bilinear form. In this setting, with the
Maslov index taking its values in the classical Witt ring~$W(k)$, the
properties above have been established in \cite{lionvergne} (over the
reals, but the proof is not different for other fields).

To deal with the general case, first note that the real part of an
anti-hermitian form~$h$, that is the map $s(x,y)= \frac{1} {2}(h(x,y) +
\overline{h(x,y)})$, is a symplectic form on the space~$V_k$ (with
scalars restricted to~$k$). Likewise, hermitian forms give rise to
symmetric, bilinear forms, thus yielding a map~$R_{K, k} : \WH(K, \sigma )\to
W(k)$ which is known to be injective (\cite{milnor}). 

Lagrangians for~$h$ are lagrangians for~$s$, and it is immediate that
the Maslov index computed in~$V$ corresponds to the Maslov index
computed in the symplectic space~$V_k$ under the map~$R_{K, k}$. It follows
that the properties (i), (ii), (iii) and (iv) hold in the general
situation as well.

The term ``cocycle condition'' is employed because the map 
$$ c : \U(V) \times \U(V) \longrightarrow \WH(K, \sigma )  $$
defined by~$c(g, h) = \maslov{\ell}{g\cdot\ell}{gh\cdot\ell}$ is then a
two-cocycle on the unitary group~$\U(V)$, for any choice of
lagrangian~$\ell$. There is a corresponding central extension : 
$$ 0 \longrightarrow \WH(K, \sigma ) \longrightarrow
\widetilde{\U(V)} \longrightarrow \U(V) \longrightarrow 1 \, ,   $$
in which the group~$\widetilde{\U(V)}$ can be seen as the
set~$ \U(V)\times \WH(K, \sigma )$ endowed with the twisted multiplication 
$$ (g, a) \cdot (h, b) = (gh, a + b + c(g, h)) \, .   $$

We conclude these definitions with a simple trick. The constructions
above, particularly the definition of the two-cocycle, involve
choosing a lagrangian in an arbitrary fashion. Moreover, the
anti-hermitian space~$V$ needs to be hyperbolic, while many spaces
arising naturally are not. Thus it is useful to note the
following. Starting with any anti-hermitian space~$(V,h)$,
put~$\D{V}=(V,-h) \oplus (V, yh)$, where the sum is
orthogonal. Then~$\D{V}$ is non-degenerate if~$V$ is, and it is
automatically hyperbolic. Indeed, for any~$g\in\U(V)$, let~$\Gamma_g$
denote its graph. Then~$\Gamma_g$ is a lagrangian in~$\D{V}$, and in
fact~$\D{V}=\Gamma_1 \oplus \Gamma_{-1}$. From now on, we will
see~$\Gamma_1$ as our preferred lagrangian. Note that there is a
natural homomorphism~$\U(V)\to \U(\D{V})$ which sends~$g$ to~$1\times
g$.

\subsection{Some two-cocycles on the braid groups} \label{subsec-conj-inv}

Let us consider a homomorphism~$r : \B_n\to \U(V)$ for some
anti-hermitian space~$V$, and let us compose it with the
map~$\U(V)\to\U(\D{V})$ just described. Let us call~$\rho : \B_n \to
\U(\D{V})$ the resulting map. We obtain a two-cocycle on~$\B_n$ by the
formula
$$ c(\beta , \gamma ) = \maslov{\Gamma_1} {\rho (\beta )\cdot
  \Gamma_1} { \rho (\beta \gamma )\cdot \Gamma_1} = \maslov{\Gamma_1}
{ \Gamma_{r (\beta )}} {\Gamma_{r (\beta \gamma )}} \, ;  $$
indeed this is the pull-back of the two-cocycle on~$\U(\D{V})$ defined
above. 

We show below that this two-cocycle must be a coboundary, so
that there must exist a map~$f : \B_n \to \WH(K)$ such that
\begin{equation} \label{eq-meyer-additive}
f( \beta \gamma ) = f(\beta ) + f(\gamma ) + c(\beta, \gamma ) \, .
\end{equation}
In other words, we shall see that~$\rho $ can be lifted to a
map~$\tilde{\rho } : \B_n\to \widetilde{\U(\D{V})}$.

A word of terminology. Given~$g, h \in\U(V)$, the particular Maslov
index 
$$ \maslov{\Gamma_1}{\Gamma_g}{\Gamma_{gh}} \, ,   $$
which involves the hyperbolic space~$\D{V}$, is often called
the {\em Meyer index} of~$g$ and~$h$. By extension, we shall also
call~$c(\beta, \gamma )$ the Meyer index of the braids~$\beta $
and~$\gamma $, with respect to~$r$. A map~$f : \B_n \to \WH(K)$
satisfying eq.\ \eqref{eq-meyer-additive} will be called {\em
  Meyer-additive} for obvious reasons.

The first thing to notice is:

\begin{prop}\label{prop-conj-inj}
Any Meyer-additive function~$f$ is conjugation-invariant.
\end{prop}

\begin{proof}
  We shall need the following simple property of Maslov
  indices. Let~$\ell_1$, $\ell_2$, $\ell_3$ be lagrangians in some
  hyperbolic, anti- hermitian space~$W$ with form~$h$. Assume
  that~$\alpha : W \to W$ is a linear map such that~$h(\alpha (x),
  \alpha (y)) = - h(x, y)$ (in other words, $\alpha$ is a
  homomorphism~$(W,h) \to (W,-h)$). Then~$\alpha \cdot \ell_i$ is a
  lagrangian in~$W$, and
$$ \maslov{\alpha \cdot \ell_1} {\alpha \cdot \ell_2}{\alpha \cdot
  \ell_3} = - \maslov{\ell_1}{\ell_2}{\ell_3} \, . $$
This is clear.

Now apply this to~$W= \D{V} = (-V) \oplus V$ and~$\alpha (x, y)= (y,
x)$. We obtain in particular 
$$ \maslov{\Gamma_{g^{-1}}} {\Gamma_1} {\Gamma_h} = -
\maslov{\Gamma_g} {\Gamma_1} {\Gamma_{h^{-1}}} \, ,   $$
for any two~$g, h\in \U(V)$. Using the formal properties of the Maslov
index, we may rewrite this 
$$ \maslov {\Gamma_1} {\Gamma_g} {\Gamma_{gh}} = \maslov {\Gamma_1}
{\Gamma_h} {\Gamma_{hg}} \, .   $$
Thus we see that the Meyer index of~$g$ and~$h$ is in fact equal to
that of~$h$ and~$g$. As a result, we see that~$f(\beta \gamma )$ is
symmetric in~$\beta, \gamma $, as we wanted.
\end{proof}

We are now in position to prove :

\begin{prop}
In the situation above, there exists a unique Meyer-additive
function~$f$ on~$B_n$ such that~$f(\sigma_i) =0$ for~$1\le i \le
n-1$. 
\end{prop}

We shall talk of the normalized Meyer-additive function associated to
the representation~$r$.

\begin{proof}
We prove the existence first. Let~$F_n$ denote the free group on~$n-1$
generators written~$\sigma_1, \ldots , \sigma_{n-1}$, so that there is
a projection map~$\pi \colon F_n \longrightarrow B_n$. We put~$R =
\ker \pi$. The two-cocycle above is certainly trivial when pulled-back
to~$F_n$, since the latter has no non-trivial central
extensions. Therefore there exists a Meyer-additive function~$\bar f
\colon F_n \to \WH(K)$; what is more, we may (and we do) impose~$\bar
f (\sigma_i )=0$. We prove now that~$\bar f(\beta )$ depends only on
the class of~$\beta \in F_n$ modulo~$R$, so that~$\bar f$ factors
through~$B_n$.

Since the representation~$r$ does factor through~$B_n$, the
two-cocycle~$c(\beta , \gamma )$ vanishes for~$\beta \in R$ and
any~$\gamma \in F_n$. It follows that~$\bar f (\beta \gamma ) = \bar
f(\beta ) + \bar f (\gamma )$ in this situation. Therefore, it
suffices to show that~$\bar f$ vanishes on~$R$. Note also that~$\bar f
(\beta \gamma ) = \bar f(\beta ) + \bar f (\gamma )$ whenever~$\beta
\gamma \in R$, for similar reasons.

The previous Proposition applies to~$\bar f$, and shows that~$\bar f$
is conjugation-invariant. Thus it is sufficient to show that~$\bar f$
vanishes on a set of generators for~$R$ {\em as a normal subgroup}. We
take for those the commutators~$[\sigma_i, \sigma_j]$ for~$| i -j| \ge
2$, and~$(\sigma_i \sigma_{i+1} \sigma_i)(\sigma_{i+1} \sigma_i
\sigma_{i+1})^{-1}$ for~$1 \le i < n-1$.

Meyer-additivity implies, as the reader will check, that~$\bar f(1) =
0$ and~$\bar f( \beta ^{-1}) = - \bar f(\beta)$. As a result~$\bar
f(\beta \gamma^{-1}) = \bar f (\beta ) - \bar f(\gamma )$
whenever~$\beta \gamma^{-1} \in R$. Therefore we have reduced the
proof to checking that~$\bar f(\sigma_i \sigma_j) = \bar f (\sigma_j
\sigma_i)$ and that~$\bar f(\sigma_i \sigma_{i+1} \sigma_i ) = \bar f(
\sigma_{i+1} \sigma_i \sigma_{i+1})$ for the relevant indices. However
$\bar f(\sigma_i \sigma_j) = \bar f (\sigma_j \sigma_i)$ trivially
holds for any pair~$i, j$ since we know that~$\bar f$ is
conjugation-invariant. 

Going back to the definitions, we see after a little calculation that
we need to prove that
\[ \m{1}{r(\sigma_i \sigma_{i+1})} {r(\sigma_i \sigma_{i+1} \sigma_i
  )} = \m{1}{r(\sigma_{i+1})} {r(\sigma_{i+1} \sigma_i \sigma_{i+1})} \, . \tag{*}  \]

In~$B_n$ there is an element~$\alpha $ such that~$\alpha \sigma_i
\alpha^{-1} = \sigma_{i+1}$ (for example~$\alpha = \sigma_1 \sigma_2
\cdots \sigma_{n-1}$). Moreover, for any~$\beta \in B_n$ we note
that~$r(\beta) \times r( \beta ) $ is an automorphism of~$\D{V}$ which
satisfies~$r(\beta) \times r( \beta ) \cdot \Gamma_{r(\gamma )} =
\Gamma_{r(\beta \gamma \beta^{-1})}$. By applying property (iv) of
Maslov indices with the automorphism~$r(\alpha )\times r (\alpha )$,
we see that we only need to prove (*) for~$ i = 1$, that is, we need
to show 
$$ \m{1}{r(\sigma_1 \sigma_{2})} {r(\sigma_1 \sigma_{2} \sigma_1
  )} = \m{1}{r(\sigma_{2})} {r(\sigma_{2} \sigma_1 \sigma_{2})} \, .  $$
Of course~$r(\sigma_1 \sigma_2 \sigma_1) = r(\sigma_2 \sigma_1
\sigma_2)$, so there are four lagrangians involved in this
equation. Appealing to the cocycle property (ii) of Maslov indices, we
obtain the equivalent equation 
\[ \m{1} {r(\sigma_1 \sigma_{2})} {r( \sigma_{2}
  )} = \m{r(\sigma_2)}  {r( \sigma_1 \sigma_{2} \sigma_1)} {r(
  \sigma_1 \sigma_{2})} \, .  \tag{**} \]
Finally consider the element~$\beta = \sigma_2^{-1}
\sigma_1^{-1}$. In~$B_n$, conjugation by~$\beta $ takes~$\sigma_2$
to~$\sigma_1$, it takes~$\sigma_1 \sigma_2 \sigma_1$ to~$\sigma_1^2
\sigma_2$, and~$\sigma_1 \sigma_2$ to itself. After applying~$r(\beta
)\times r(\beta )$ to the right hand side of (**), one
obtains thus~$\m{r(\sigma_1)} {r(\sigma_1^2 \sigma_2)}  {r(\sigma_1
  \sigma_2)}$. Now apply~$1 \times r(\sigma_1^{-1})$ and you get the
left hand side of (**). This concludes the proof of the existence
of~$f$. 

We turn to the uniqueness of Meyer-additive functions. Such a map is
clearly determined by its values on the generators~$\sigma_i$ of the
braid group. What is more, these generators are all conjugate, as we
have seen. Thus~$f$ is determined by, say, the value $f(\sigma_1)$,
and there can only be one Meyer-additive function vanishing
on~$\sigma_1$, which is stronger than the statement in the Proposition.
\end{proof}

\subsection{The Markov conditions}

As announced at the beginning of this section, one can hope to produce
a Markov function by using a sequence of representations~$r_n : \B_n
\to \U(V_n)$, and using for~$f_n$ the corresponding normalized,
Meyer-additive function, whose coboundary is the two-cocycle~$c_n$.
The chief example seems to be the Burau representation, to be
described next. Other (unsuccessful) attemps by the authors include
the Lawrence-Krammer-Bigelow representation, the representations
afforded by Hecke algebras, and those related to the modules of the
quantum group~$U_q(\mathfrak{sl}_2)$.

It is easy to write down the conditions for the functions~$f_n$ to
combine into a Markov function. Let us do this now in the special
case, covering the Burau representation, when
$$ V_{n+1} = V_n \oplus (triv)  $$
as~$\B_n$-modules, where~$(triv)$ refers to a
trivial~$\B_n$-module. The additivity of the Maslov index immediately
implies that~$c_{n+1}(\beta , \gamma ) = c_n(\beta , \gamma )$
for~$\beta, \gamma \in\B_n$ (here we see~$\B_n$ as a subgroup
of~$\B_{n+1}$, suppressing any inclusion map from the notation). It
follows that~$f_{n+1}$, when restricted to~$\B_n$, coincides
with~$f_n$.

The collection~$(f_n)$ is then a Markov function if and only if 
\begin{equation} \label{eq-markov-condition}
c_{n+1}( \beta , \sigma_n^\pm) = \maslov{\Gamma_1} {\Gamma_{r_{n+1}(\beta)} }
{\Gamma_{r_{n+1}(\beta \sigma_n^\pm)} }= 0 \, ,
\end{equation}
for~$\beta \in\B_n$. 

\begin{rmk} \label{rmk-cofinal}
It may (and it will) happen that we have at our disposal a collection
of representations~$r_{d_n} \colon \B_{d_n} \to \U(V_{d_n})$ for some
increasing sequence of integers~$d_1, d_2, \ldots $, but that~$r_n$ is
not initially defined for all~$n$. In this case, given an integer~$n$
we shall pick the smallest~$d_m$ such that~$n \le d_m$, and
define~$r_n$ to be the composition of the inclusion~$\B_n \to
\B_{d_m}$ followed by~$r_{d_m}$. 

(In practice this will happen with the Burau representation at~$t=-1$,
for which the anti-hermitian form is only non-degenerate
for~$\B_{2n}$; so for~$\B_{2n-1}$ we have to consider its inclusion
into~$\B_{2n}$.) 
\end{rmk}

%% file: facile.tex
\section{The case of the Burau representation} \label{sec-burau}

\subsection{Definitions}

Initially, the Burau representation is the homomorphism~$\B_n\to GL_n(\Z[t, t^{-1}])$
mapping~$\sigma_i$ to the matrix 
$$ \Sigma_i = \left(\begin{array}{cccc}
I_{i-1} & 0 & 0 & 0 \\
0 & 1 - t & 1 & 0 \\
0 & t & 0 & 0 \\
0 & 0 & 0 & I_{n-i-1}
\end{array}\right) \, . 
$$
(Some authors use the transpose of this matrix, for example in
\cite{kassel}.)

The ring~$\Z[t, t^{-1}]$ has an involution~$\sigma _0$
with~$\sigma_0(t)= t^{-1}$. As above we write~$\bar x$ instead
of~$\sigma_0 (x)$. Now put 
$$ \Omega_n = \left(\begin{array}{ccccc}
1 & 0 & 0 & \cdots & 0 \\
1 - t & 1 & 0 & \cdots & 0 \\
1-t & 1-t & 1 & \cdots & 0 \\
\vdots & \vdots & \vdots & \ddots & \vdots \\
1-t & 1-t & 1-t & \cdots & 1
\end{array}\right) \, . 
$$
For each~$i$, $1 \le i < n$, one has ${}^T \overline{\Sigma_i} \, \Omega_n\,
\Sigma_i = \Omega_n$, where~${}^T A$ denotes the transpose of the
matrix~$A$ (see \cite{kassel}, Theorem 3.1). 

Now let~$K$ be a field with involution~$\sigma $, and pick a
homomorphism~$\alpha : \Z[t, t^{-1}] \to K$ compatible with~$\sigma_0$
and~$\sigma $. Typical examples will be: (i) $K = \Q(t)$ or~$\F_p(t)$
with involution~$\sigma $ defined by~$\sigma (t)=t^{-1}$ and~$\alpha
(t) = t$, and (ii) $K=\Q$ or~$\R$ or~$\F_p$ with trivial involution
and~$\alpha (t) = -1$.

We let~$V_n= K^n$ and view it as a~$\B_n$-module by applying~$\alpha $
to the coefficients of the Burau representation. Likewise, we may
see~$\Omega_n$ as a matrix with coefficients in~$K$. Now we put
$$ H_n = \Omega_n -  \, {}^{T} \overline{ \Omega_n} \, ,   $$
so that~${}^T H_n = -\, \overline{ H_n}$, and
also~${}^T \overline{ \Sigma_i} \, H_n\, \Sigma_i = H_n$.

The space~$V_n$ is thus anti-hermitian when equipped with the
form~$h_n$ given by (identifying vectors of~$V_n$ with~$n\times 1$
matrices):
$$ h_n(x,y)= {}^T x\, H_n \, y \, .  $$
We are interested in cases when this form is non-degenerate. A simple
calculation leads to 
\begin{lem}
The determinant of~$H_n$ is given by
$$ \det H_n = (-1)^n \left[ (1-\alpha (t))^{n-1}   -   (\frac{1}
  {\alpha (t)} -1 )^{n-1} \right] \, .   $$
In particular, when~$\alpha (t) \ne 1$, at most one of~$\det H_n$
and~$\det H_{n+1}$ can be zero.
\end{lem}

From now on we assume that~$\alpha (t)\ne 1 $ so that this determinant
is non-zero for infinitely many values of~$n$. 

The form~$h_n$ is preserved by~$\B_n$, so we end up with a map~$r_n :
\B_n \to \U(V_n)$, as requested in the previous section (and bearing
Remark~\ref{rmk-cofinal} in mind). Thus we have maps
$$ f_n : \B_n \longrightarrow \WH(K, \sigma  ) \, ,   $$
and we shall prove presently that together they give a Markov
function, and thus a link invariant.

Our first step is to prove the existence of an auxiliary Markov
function:

\begin{prop}\label{prop-petit-markov}
For each~$n \ge 2$ and each~$\beta \in \B_n$, put 
$$ d_n^K(\beta ) = d_n (\beta )= \dim_K \ker( r_n(\beta ) - Id_n) \, .   $$
Then~$(d_n)_{n\ge 2}$ is a Markov function.
\end{prop}

Note that the hermitian structure does not come into play here. Also,
the characteristic of~$K$ may very well be~$2$ in this Proposition.

\begin{proof}
  That~$d_n$ is conjugation-invariant is obvious. We claim that,
  for any~$n\times n$ matrix~$M$, the rank of~$\widetilde M
  \Sigma_n^\pm - Id_{n+1}$ is one more than the rank of~$M - Id_n$,
  where~$\widetilde M$ is the~$(n+1)\times(n+1)$ matrix obtained from~$M$
  by adding a~$1$ in the bottom right corner. This will imply the
  result.

  It is enough to prove the claim with~$\Sigma_n$ replaced by its
  transpose~${}^T \Sigma_n$. We observe that, adding the last column
  of~$\widetilde M {}^T\Sigma_n - Id_{n+1}$ to the one immediately on its
  left, we obtain 
$$ \left(\begin{array}{cc}
M - Id_n & * \\
0            & -1
\end{array}\right) \, . 
$$
This takes care of~$\Sigma_n$. The same operation on~$\widetilde M
  {}^T\Sigma_n^{-1} - Id_{n+1}$ gives
$$ \left(\begin{array}{cc}
M - Id_n & * \\
0            & \frac{-1} {t}
\end{array}\right) \, . \qedhere
$$
\end{proof}

\subsection{The main theorem}

Before we turn to the proof of the main result, we quote Thomas's
criterion for the vanishing of Maslov indices:

\begin{lem}[Thomas] \label{lem-thomas}
Let~$\ell_1$, $\ell_2$ and~$\ell_3$ be lagrangians in the anti-hermitian space~$U$. If 
$$ \dim(\ell_1\cap\ell_2) + \dim(\ell_2\cap\ell_3) +
\dim(\ell_3\cap\ell_1) =  \dim \ell_1 + 2 \dim(\ell_1 \cap
\ell_2 \cap \ell_3) \, , 
$$
then 
$$ \maslov {\ell_1} {\ell_2} {\ell_3} = 0 \, .   $$
\end{lem}
(Of course~$\dim \ell_1 = \dim \ell_2 = \dim \ell_3 = \frac{1} {2}
\dim U$, so the equality is symmetric in~$\ell_1$, $\ell_2$,
$\ell_3$.) For a proof see \cite{thomas}, Proposition 4.1. We may now state 

\begin{thm} \label{thm-main}
For each~$n\ge 2$, let~$r_n : \B_n \to \U(V_n)$ be the homomorphism
obtained from the Burau representation as above. 

There is a unique map
$$ f_n : \B_n \longrightarrow \WH( K, \sigma  )  $$
with the property that~$f_n(\sigma_i)= 0$ and 
$$ f_n(\beta \gamma ) = f_n(\beta ) + f_n(\gamma ) + \maslov{\Gamma_1}
{\Gamma_{r_n(\beta )} } {\Gamma_{r_n(\beta \gamma )}} \, ,   $$
for~$\beta, \gamma \in \B_n$.

The collection~$(f_n)_{n\ge 2}$ is a Markov function.
\end{thm}

Given a link~$L$ which is isotopic to the closure~$\hat \beta $ of the
braid~$\beta \in \B_n$, the invariant~$f_n(\beta )$ will be
written~$\Theta_K (L)$ (the dependence on~ $\sigma $ and $\alpha $
being implicit).

\begin{proof}
At this stage, it remains to prove 
$$ \maslov{\Gamma_1} {\Gamma_{r_{n+1}(\beta)} }
{\Gamma_{r_{n+1}(\beta \sigma_n^\pm)} }= 0  $$
  for all~$\beta \in \B_n$, see eq.\ \eqref{eq-markov-condition}.  

  We use Thomas's criterion (Lemma \ref{lem-thomas}). Note that, as
  explained in Remark~\ref{rmk-cofinal}, the representation~$r_{n+1}$
  is the composition of an inclusion~$\B_{n+1} \to \B_N$ for a
  certain~$N$, followed by the Burau representation of~$\B_N$. However
  for the sake of applying Thomas's lemma, all we need is to prove an
  equality involving the dimensions of certain subspaces made up from
$$ \ell_1 = \Gamma_1, \quad \ell_2 = \Gamma_{r_{n+1}(\beta)} , \quad
\ell_3 = \Gamma_{r_{n+1}(\beta \sigma_n^\pm)} \, .  $$
For this, we can and we will assume that~$N=n+1$ and so that~$r_{n+1}$
is really the Burau representation of~$\B_{n+1}$ : indeed going into a
larger braid group only adds a common summand to the subspaces in
sight, which does not affect the equality of dimensions to be
proved. So we work with the space~$\D{V _{n+1}}$ of dimension~$2n+2$.

  Let us identify the terms in the lemma. First the
  intersection~$\Gamma_1 \cap \Gamma_{r_{n+1}(\beta) }$ is isomorphic, as a
  vector space, with~$\ker(r_{n+1}(\beta) - Id)$. So its dimension
  is~$d_{n+1}(\beta )$, with the terminology of Proposition
  \ref{prop-petit-markov}. In turn, $d_{n+1} (\beta ) = 1 + d_n(\beta
  )$, clearly. 

  Proceeding in a similar fashion with the other terms, we see that
  the equality to check is really
\begin{multline*}
1 + d_n( \beta ) + d_{n+1}( \beta \sigma_n^\pm ) +
d_{n+1}(\sigma_n^\pm) = \\ n+1
+ 2 \dim (\ker(r_{n+1}(\beta) - Id) \cap \ker(r_{n+1}(\sigma_n^\pm) - Id)) \, . 
\end{multline*}
For simplicity let us now write~$\sigma_{n}^\pm$
for~$r_{n+1}(\sigma_n^\pm)$.  It is easy to
describe~$\ker(\sigma_n^\pm - Id)$. Indeed, let~$v_n \in V_n$ be the
vector with coordinates~$(1, t, t^2, \ldots, t^{n-1})$; it is readily
checked that~$\gamma v_n = v_n$ for all~$\gamma \in\B_n$. We see
that~$$\ker(\sigma_n^\pm -Id)= V_{n-1} \oplus K\, v_{n+1} \, ;$$ in
particular it has dimension~$n$.

What is more, we see that the dimension of~$\ker(r_{n+1}(\beta ) - Id)
\cap \ker(\sigma_n^\pm -Id)$ is one more than the dimension
of~$\ker(r_{n+1}(\beta ) - Id) \cap V_{n-1}$. Comparing with the
decomposition~$V_n = V_{n-1} \oplus K\, v_n$, we conclude that 
$$ \dim \ker(r_{n+1}(\beta ) - Id)
\cap \ker(\sigma_n^\pm -Id) = \dim \ker(r_n(\beta ) - Id) = d_n(\beta
) \, .  $$

Putting together these elementary observations, we find that Thomas's
criterion reduces to 
$$ d_n(\beta ) + d_{n+1}(\beta \sigma_n^\pm ) = 2 d_n(\beta ) \, ,   $$
which is guaranteed by Proposition \ref{prop-petit-markov}. This concludes
the proof.
\end{proof}

%% file: examples.tex
\section{Examples \& Computations} \label{sec-examples}

\subsection{Signatures} \label{subsec-signatures}

Let us start with the example of~$K=\R$, with trivial involution,
and~$\alpha  (t) = -1$. The above procedure yields a link invariant
with values in~$W(\R) \cong \Z$ (the isomorphism being given by the
signature of quadratic forms).

However, Gambaudo and Ghys have proved in \cite{ghys} that the
invariant which is classically called {\em the signature of a link} is
in fact given by a normalized, Meyer-additive Markov function (in our
terminology). As a result, ~$\Theta_\R(L)$ must always coincide with
the signature of~$L$.

An obvious refinement is obtained by taking~$K= \Q$ (and still~$\alpha
(t)=-1$). The invariant~$\Theta_\Q(L)$ lives in~$W(\Q)$. Recall
from~\cite{milnor} that there is an exact sequence
$$ 0 \longrightarrow W(\Z) \longrightarrow W(\Q) \longrightarrow
\bigoplus_{p} W(\F_p) \longrightarrow 0 \, .   $$
This sequence is split by the homomorphism~$W(\Q) \to W(\R) \cong
W(\Z)$. Moreover, for~$p$ odd the group $W(\F_p)$ is
either~$\Z/2\times \Z/2$ or~$\Z/4$ according as~$p$ is~$1$ mod~$4$ or
not, while~$W(\F_2) = \Z/2$. We obtain the invariants alluded to in
the introduction.


We need not restrict ourselves to the case~$\alpha (t)= -1$,
however. For example we may take~$K= \C$ with the usual complex
conjugation, and~$\alpha (t)= \omega$, a complex number of
module~$1$. All of the above generalizes. We obtain a link invariant
with values in~$\WH(\C)\cong \Z$, whose value on~$L$ will be
written~$\Theta_\omega(L)$.

When~$\omega$ is a root of unity at least, Gambaudo and Ghys also
prove in {\em loc.\ cit.} that the so-called {\em Levine-Tristram
  signature} of a link is given by a normalized, Meyer-additive
function, so that it must agree with~$\Theta_\omega(L)$. Again, we
obtain a refinement. Whenever~$\omega$ is algebraic, the field~$K=
\Q(\omega)$ is a number field. There is an exact sequence
\[ 0 \longrightarrow W(\mathcal{O}) \longrightarrow W( K )
\longrightarrow  \bigoplus_\mathfrak{p} W(\mathcal{O} /
\mathfrak{p}) \longrightarrow 0 \, ,  
  \]
  where~$\mathcal{O}$ is the ring of integers in~$K$, and the direct
  sum runs over the prime ideals~$\mathfrak{p}$. Thus the
  invariant~$\Theta_K(L) \in \WH(K) \subset W(K)$ yields invariants in
  the Witt rings of the various fields~$\mathcal{O}/\mathfrak{p}$,
  which are finite.

\subsection{The case~$K= \Q(t)$}
Our favorite example is that of~$K= \Q(t)$ with~$\sigma (t)= t^{-1}$
and~$\alpha (t)= t$; in some sense we shall be able recover the
signatures of the previous examples from this one. We shall go into
more computational considerations than above. The reader who wants to
know more about the technical details should consult the accompanying
{\sc Sage} script, available on the authors' webpages. Conversely,
this section is a prerequisite for understanding the code.

Consider~$\beta = \sigma_1^3 \in \B_2$ as a motivational
example. Here~$L= \hat \beta $ is the familiar trefoil knot. 

\medskip

\begin{minipage}{0.07\textwidth}
\includegraphics[width= \textwidth]{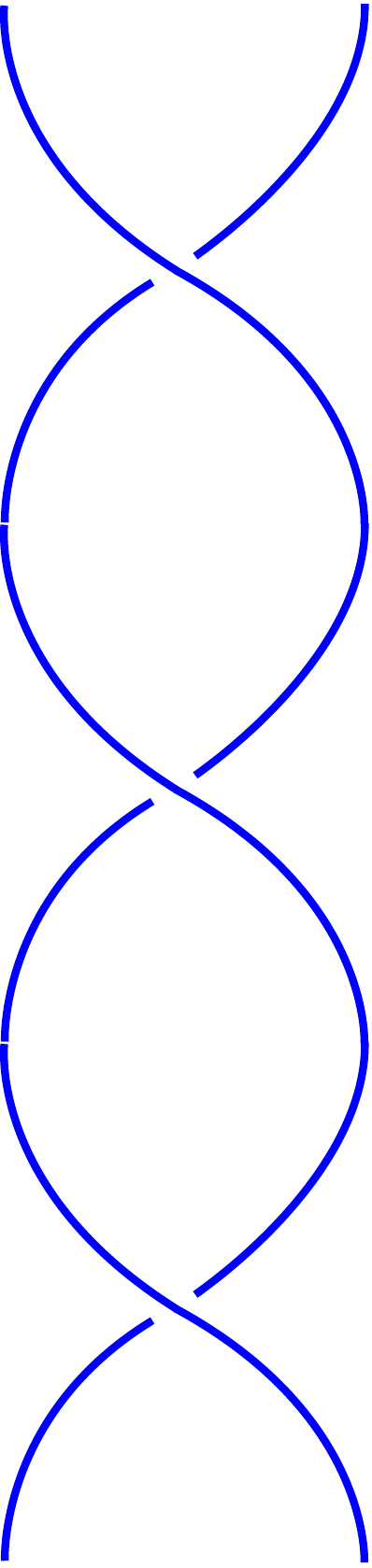}
\end{minipage} \hspace{3em} \begin{minipage}{0.35\textwidth}
\footnotesize

Left: the braid~$\sigma_1^3$.

 Right: its closure, the trefoil knot.
\end{minipage} \hspace{3em} \begin{minipage}{0.3\textwidth}
\includegraphics[width= \textwidth]{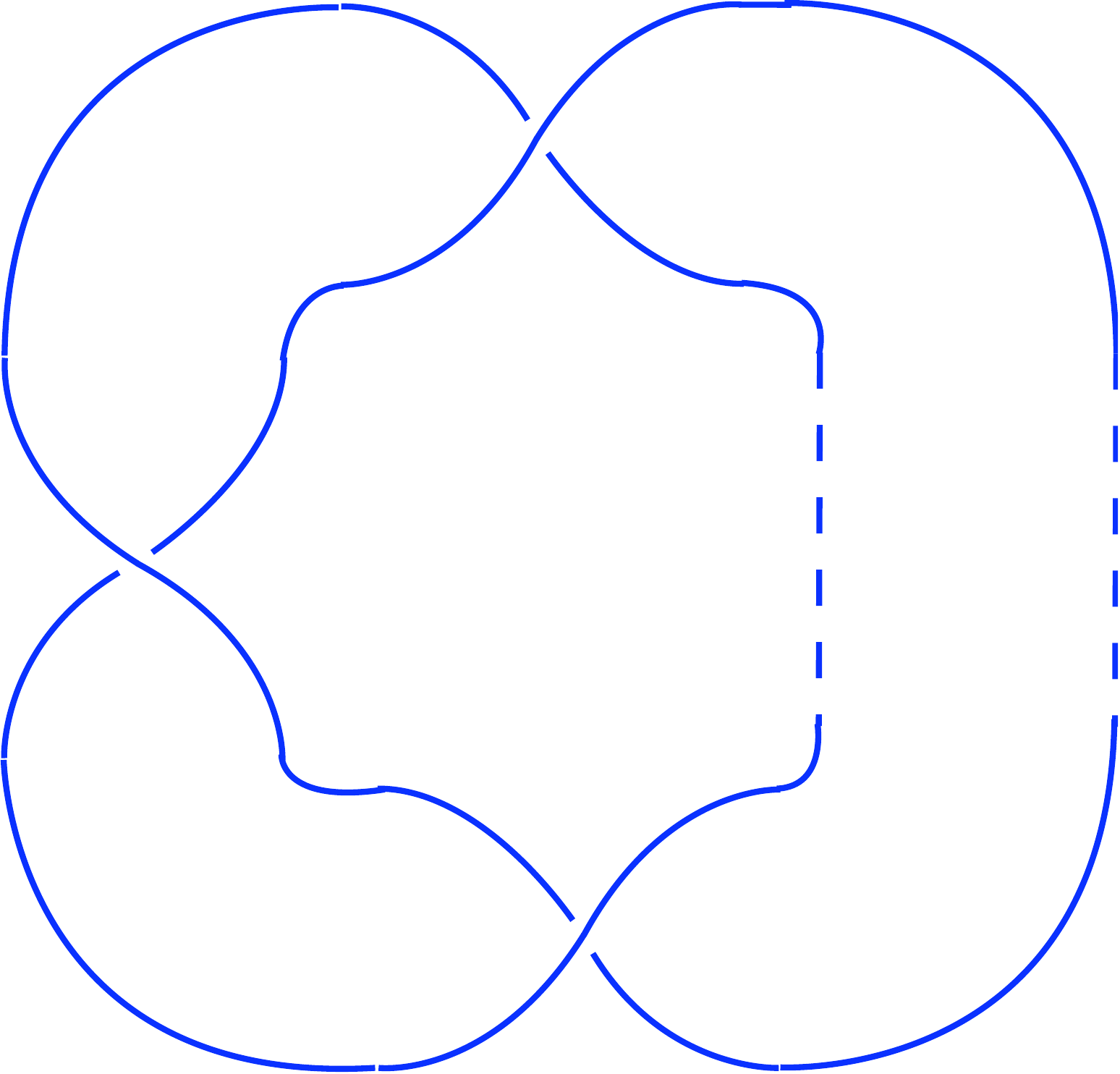}
\end{minipage}~

\medskip

When computing~$\Theta_{\Q(t)}(L)$ we are led to perform additions
in~$\WH(\Q(t), \sigma )$. Since a hermitian form can always be
diagonalized, we can represent any element in the hermitian Witt ring
by a sequence of scalars. In turn, these are in fact viewed
in~$k^\times / N(K^\times)$, where as above~$k= K^\sigma$ and~$N: K\to
k$ is the norm map~$x \mapsto x \bar x$. Summing two elements amounts
to concatenating the diagonal entries.

Let us turn to the example of the trefoil knot. We relax the notation,
and write~$f$ for~$f_n$ when~$n$ is obvious or irrelevant, and we
write~$c$ for the two-cocycle~$c(\beta , \gamma )= \maslov{\Gamma_1}
{\Gamma_{r(\beta )}} {\Gamma_{r(\beta \gamma )}}$, so we have the
formula~$f(\beta \gamma ) = f(\beta ) + f(\gamma ) + c(\beta ,
\gamma)$. Now:
\begin{align*}
  \Theta_{\Q(t)}(L) = f(\sigma_1^3) & = f(\sigma_1) + f(\sigma_1^2) +
  c(\sigma_1, \sigma_1^2) \\
  & =f(\sigma_1) + \left( f(\sigma_1) + f(\sigma_1) + c(\sigma_1,
    \sigma_1)\right) + c(\sigma_1,
  \sigma_1^2) \\
  & = 0 + 0 + 0 + c(\sigma_1, \sigma_1) + c(\sigma_1, \sigma_1^2) \,
  . 
\end{align*}
Thus~$\Theta_{\Q(t)}(L)$ is the sum of two Maslov indices, and direct
computation shows that it is represented by
$$ \left[1, -1, \frac{-2t^{2} + 2t - 2}{t}, 1, -1, 2\right] \, .  $$
Now, the hermitian form given by the matrix 
$$ \left(\begin{array}{rr}
-1 & 0 \\
0 & 1
\end{array}\right)
$$
is hyperbolic and so represents the trivial element in the Witt ring. We
conclude that~$\Theta_{\Q(t)}(L)$ is represented by the form whose
matrix is 
$$ \left(\begin{array}{rr}
\frac{-2t^{2} + 2t - 2}{t} & 0 \\
0       & 2
\end{array}\right) \, . 
$$

Comparing elements in the Witt ring can be tricky. For example, we
need to be able to tell quickly whether this last form is actually~$0$
or not. In general, link invariants need to be easy to compute and
compare. 

To this end, we turn to the construction of a Laurent polynomial
invariant. There is a well-known homomorphism~$D : \WH(K, \sigma )\to
k^\times / N(K^\times)$ given by the {\em signed determinant} : given
a non-singular, hermitian~$n \times n$-matrix~$A$ representing an
element in the Witt ring, then~$D(A)= (-1)^{\frac{n(n-1)} {2}} \det(A)$. This
defines a link invariant with values in $k^\times / N(K^\times)$, and
for the trefoil we have
$$ D( \Theta_{\Q(t)}(L)) = \frac{-t^{2} + t -1 }{t} \, .   $$
(Note how we got rid of the factor~$4= N(2)$.) This happens to be the
Alexander-Conway polynomial of~$L$; see the next section for more on
this. 

For the sake of practicalities, let us indulge in some
computational details:

\begin{lem}
Any element in~$k^\times / N(K^\times)$ can be represented by a
fraction of the form 
$$ \frac{D(t)} {t^d} \, , 
$$
where~$D(t)$ is a polynomial in~$t$, of degree~$2d$, not divisible
by~$t$, and which is also palindromic.

What is more, if~$D$ has minimal degree among such polynomials,
then it is uniquely defined up to a square in~$\Q^\times$.
\end{lem}

The proof will also indicate an algorithm to compute the minimal~$D$. 

\begin{proof}
By definition any element is represented by 
$$ \frac{F(t)} {G(t)} \, , 
$$
for some polynomials~$F$ and~$G$, so that we may multiply by the
norm~$G(t)G(t^{-1})$ to obtain a Laurent polynomial
representative. Since it must be stable under~$\sigma $, it has the
form 
$$ \frac{R(t)} {t^d}
$$
with~$R$ palindromic of degree~$2d$, not divisible by~$t$. 

We turn to the uniqueness. Given a polynomial~$P$, consider~$\tilde P
= t^{\deg P} P(t^{-1})$, which is again a polynomial. In this notation
we have~$R = \tilde R$. The assignment~$P\mapsto \tilde P$ is
multiplicative, that is~$\widetilde{PQ} = \tilde P \tilde Q$;
moreover, performing this operation twice on a polynomial~$P$ not
divisible by~$t$ gives again~$P$ (while in general the power of~$t$
dividing~$P$ disappears from~$\tilde{\tilde{P}}$, as follows
from~$\tilde{t} = 1$). We conclude that, when~$P$ is an irreducible
polynomial, prime to~$t$, then~$\tilde P$ is also irreducible.

Now factor~$R = \prod P_i^{\alpha_i}$ into a product of powers of
prime polynomials. For a given~$P_i$, we have one of two options. The
first possibility is that~$\tilde P_i$ is prime to~$P_i$ (for example
if~$P_i = t^2 - 2$), so that~$R = \tilde R$ is divisible
by~$P_i\tilde P_i$. In this case we divide~$R / t^d$ by the norm
of~$P_i^{\alpha_i}$, which has the effect of replacing~$R$ by a
polynomial of smaller degree with all the same properties as~$R$. Do
this for all such factors~$P_i$.

The remainding factors~$P_i$ of the new~$R$ are all of the second
type, that is~$\tilde P_i$ is a scalar multiple of~$P_i$ (for example~$P_i=
t-1$). Write~$\alpha_i = 2v_i + \varepsilon_i$ with~$\varepsilon_i=0$
or~$1$, and divide~$R$ by the norm of~$P_i^{v_i}$. Do this for all the
factors.

There results a polynomial, which we still call~$R$, with the same
properties as above, and with the extra feature that it factors as a
product~$R=\prod P_i$ where each~$P_i$ is a scalar multiple of~$\tilde
P_i$ (the various~$P_i$'s being pairiwise coprime). We now show that
this~$R$ can be taken for~$D$.

Let~$Q$ be any polynomial such that~$Q/ t^q$ represents the same
element in~$k^\times / N(K^\times)$ than~$R/t^d$. We prove that~$R$
divides~$Q$, which certainly implies the uniqueness statement.

Indeed there must exist coprime polynomials~$F$ and~$G$ and an
integer~$k$ such that
$$  R G \tilde G = t^k Q F \tilde F \, .  $$
Pick a prime factor~$P_i$ in~$R$. If~$P_i$ does not divide~$Q$, then
it must divide one of~$F$ or~$\tilde F$; hence~$\tilde P_i$ divides
the other one. However since~$\tilde P_i = P_i$ up to a scalar, we see
that~$P_i^2$ divides~$QF\tilde F$, so that~$P_i^2$ divides~$RG\tilde
G$. We know that~$P_i^2$ does not divide~$R$, so~$P_i$ divides one
of~$G$ or~$\tilde G$, hence both. We see that~$F$ and~$G$ have a
factor in common, a contradiction which shows that~$P_i$ divides~$Q$.
\end{proof}

Here is another simple invariant deduced from~$\Theta
_{\Q(t)}$. Suppose~$\theta $ is a real number such that~$e^{i \theta
}$ is not algebraic (this excludes only countably many possibilities
for~$\theta$). The assignment~$t\mapsto e^{i \theta }$ gives a field
homomorphism~$\Q(t) \to \C$ which is compatible with the involutions
(on the field of complex numbers we use the standard
conjugation). There results a map
$$ \WH(\Q(t), \sigma ) \longrightarrow \WH(\C) \cong \Z \, ,  $$
which we call the~$\theta $-signature. It is clear by construction
that it agrees with the invariant~$\Theta_{e^{i \theta }}$ presented
in \S\ref{subsec-signatures}.

Looking at the trefoil again, we obtain the form over~$\C$
$$ \left(\begin{array}{rr}
2- 4 \cos(\theta )  & 0 \\
0                & 2
\end{array}\right)
$$
whose signature is~$0$ if~$0 < \theta < \frac{\pi} {3}$ and~$2$
if~$\frac{\pi} {3} < \theta < \pi$ (the diagonal entries are always
even functions of~$\theta $, so we need only consider the values
between~$0$ and~$\pi$.) We may present this information with the help
of a camembert : 
%
%
%
%
%

\begin{minipage}{.5\textwidth}
\includegraphics[width=\textwidth]{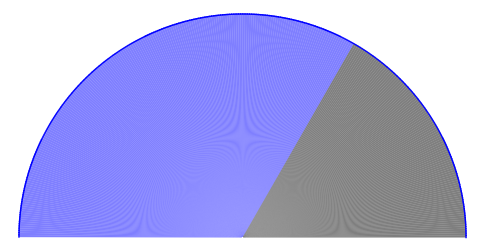}
\end{minipage}~\begin{minipage}{.5\textwidth}
\begin{tabular}{cccc}
0 & 2  \\
\includegraphics[width=0.07\textwidth]{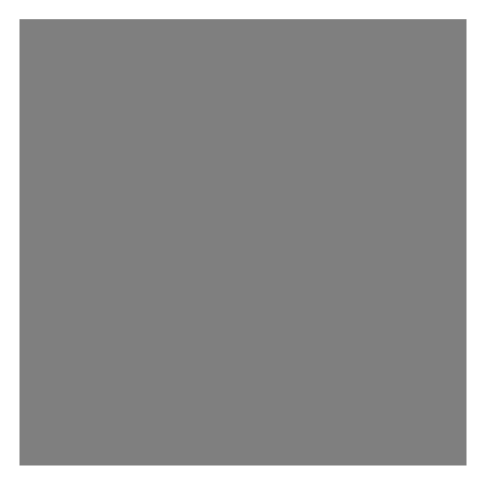}  & 
\includegraphics[width=0.07\textwidth]{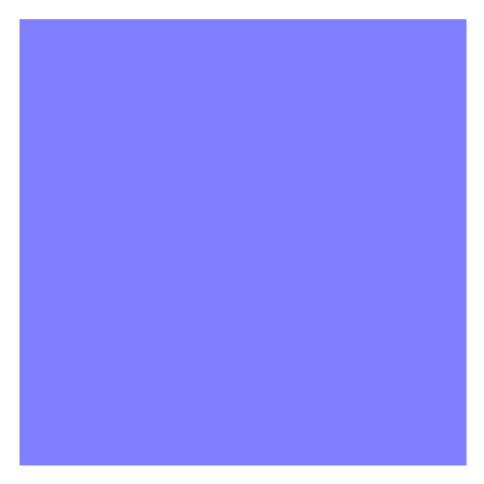}  
\end{tabular}
\end{minipage}

This figure is a link invariant.

We may get rid of the restriction on~$\theta $. Given an element
in~$\WH(\Q(t))$, pick a diagonal matrix as representative, and arrange
to have Laurent polynomials as entries. Now substitute~$e^{i \theta }$
for~$t$, obtaining a hermitian form over~$\C$, and consider the
function which to~$\theta $ assigns the signature of this form. This
is a step function~$s$, which is even and $2\pi$-periodic.

Now, whenever~$\theta $ is such that~$e^{i \theta }$ is not algebraic,
the value~$s(\theta )$ is intrinsically defined by the procedure
above, and thus does not depend on the choice of representative. Since
such~$\theta $ are dense in~$\R$, the following is well-defined: 
\[ \hat s(\theta ) = \lim_{\alpha \to \theta, \alpha > \theta }
s(\alpha ) \, .   \]
At least this provides a definition for all~$\theta $, though it is
not so easy to work with it. This may well change in the future when
we prove that it is possible to work with the ring~$\z[\frac{1} {2},
  t, t^{-1}]$ rather than the field~$\Q(t)$. This will prove that, at
least when~$\theta $ is not a ``jump'' for~$s$, the value~$\hat s
(\theta )$ agrees with~$\Theta_{e^{i \theta}}$.

To give a more complicated example, take~$\sigma_{1}^{3}
\sigma_{2}^{-1} \sigma_{1}^{2} \sigma_{3}^{1} \sigma_{2}^{3}
\sigma_{3}^{1}\in B_{4}$. The braid looks as follows:

\begin{center}
\includegraphics[width= 0.6\textwidth]{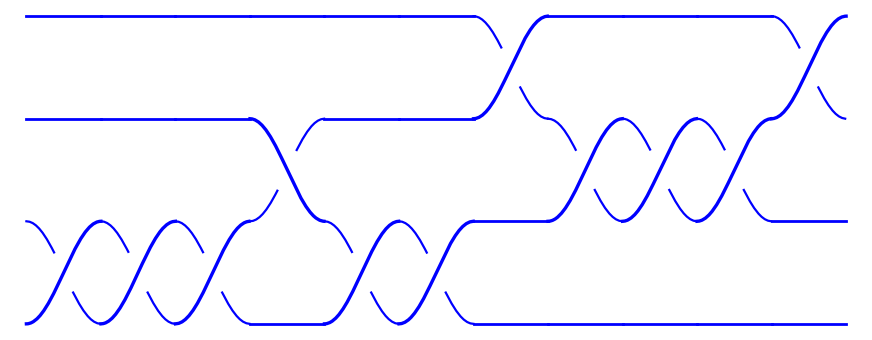}
\end{center}

The signed determinant is 
$$ \frac{3t^{6} - 9t^{5} + 15t^{4} - 17t^{3} + 15t^{2} - 9t +
  3}{t^{3}} \, ,   $$
where the numerator has minimal degree.

The camembert is

\begin{minipage}{.5\textwidth}
\includegraphics[width=\textwidth]{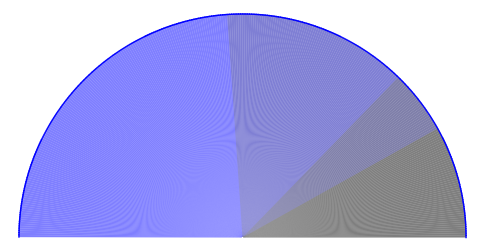}
\end{minipage}~\begin{minipage}{.5\textwidth}
\begin{tabular}{cccc}
0 & 2 & 4 & 6 \\
\includegraphics[width=0.07\textwidth]{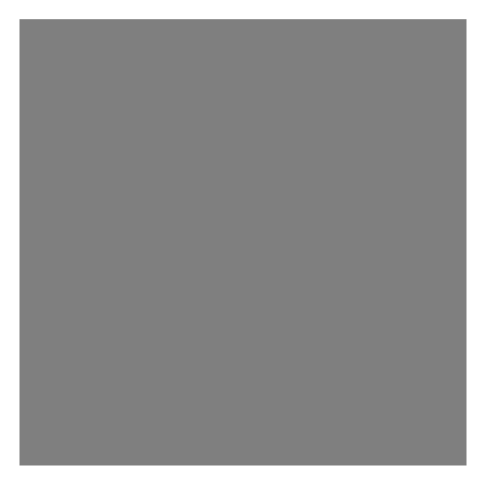}  & 
\includegraphics[width=0.07\textwidth]{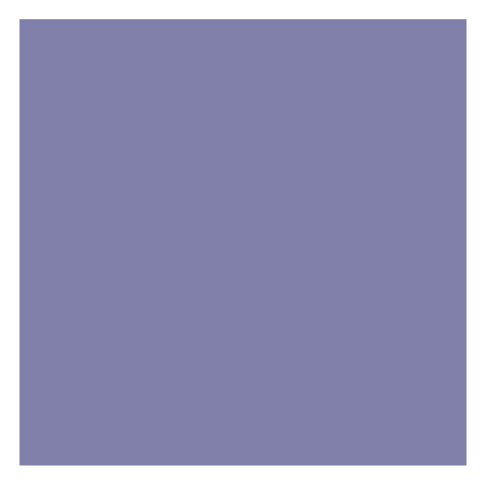} & 
\includegraphics[width=0.07\textwidth]{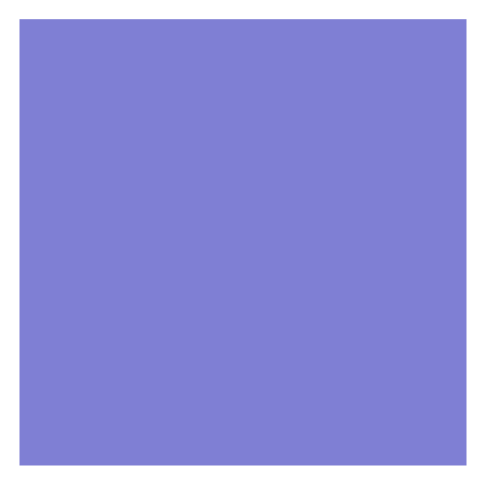}  &
\includegraphics[width=0.07\textwidth]{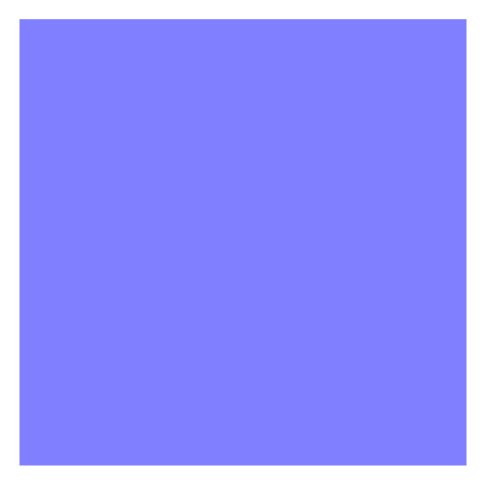} 
\end{tabular}
\end{minipage}

On the authors' webpages, the reader will find many examples of links
for which the corresponding camemberts and polynomials are given.

\subsection{Comparison with the Alexander-Conway polynomial}

Let us expand a little on the above ``coincidence'' in the case
when~$L$ is the trefoil knot, for which~$D(\Theta_{\q(t)}(L))$ happens
to be equal to the Alexander-Conway polynomial of~$L$. And first,
since there are several (related) polynomials with that name, let us
add that we consider the polynomial~$\nabla_L(s)$ as defined
in~\cite{kassel}, \S3.4.2. It is the only link invariant with values
in~$\z[s, s^{-1}]$ satisfying the skein relation 
\[ \nabla_{L_+}(s) - \nabla_{L_-}(s) = (s^{-1} - s) \nabla_{L_0}(s) \,
,   \]
in standard notation. It follows from this definition
that~$\nabla_{L}(s) = \nabla_L(-s^{-1})$, since~$L \mapsto
\nabla_{L}(-s^{-1})$ is a link invariant satisfying the same skein
relation (and taking the value~$1$ on the unknot).

It is well known that when~$L$ is a knot, rather than just a link, one
has~$\nabla_L(s) = P_L(s^2)$ for some Laurent polynomial~$P_L$. It
follows that~$P_L(t) = P_L(t^{-1})$. Thus~$P_L$ defines an element
in~$k^\times / N(K^\times)$, and we may compare it
with~$D(\Theta_{\q(t)}(L))$.

We have carried a computer experiment, whose result is that {\em for
  all the knots we have tested, the Alexander-Conway polynomial
  and~$D(\Theta_{\q(t)}(L))$ are equal in~$k^\times /
  N(K^\times)$}. The experiment was conducted on 157 (pairwise
non-isotopic) knots. It is of course reasonable to conjecture that
this holds for all knots, though we do not have a proof of this fact.



%% file: weil.tex
\section{Examples related to the Weil representation} \label{sec-weil}

\subsection{Background on the Weil representation} \label{subsec-background}
Consider~$\Sp_{2n}(\R)$, the symplectic group over the reals. Its
fundamental group is~$\Z$, so this group has a twofold cover, usually
called the metaplectic group and denoted by~$M_{2n}(\R)$.

The metaplectic group is famous for having a semi-simple
representation on the Hilbert space~$L^2(\R^n)$, called the Weil
representation; we refer to \cite{lionvergne} for a description. We
are chiefly interested in the fact that this representation has a {\em
  trace}, in the following sense. Let us write~$T (g)$ for the operator
corresponding to~$g$. For any smooth function~$f$ with compact support
on~$M_{2n}(\R)$, we may define an operator~$\bar T (f)$ by the obvious
averaging process, that is
$$ \bar T (f) = \int_{M_{2n}(\R)} f(g) T (g) dg \, .   $$
(Here a unimodular Haar measure is employed.) This operator has a
trace in the naive sense, namely for any Hilbert basis~$(e_r)_{r
  \in\N}$, one has
$$ \sum_r \langle \bar T (f) e_r,  e_r \rangle  < +\infty \, .   $$
What is more, the sum above can be computed as the integral 
$$ \int_{M_{2n}(\R)} f(g) \theta_n (g) \, dg \, .   $$
where~$\theta_n $ is a smooth function defined on a dense, open set.

In \cite{thomas}, Thomas gives a description of the map~$\theta_n $,
which we partially reproduce below. It turns out that~$\theta_n^{-1}$
extends to a continuous function on the whole of~$M_{2n}(\R)$, which
is conjugation-invariant.

Since our construction in the case~$K=\R$, $\alpha (t)=-1$ already
considered gives a map~$r_{2n} : \B_{2n} \to \Sp_{2n}(\R)$, and since the
cohomological considerations of \S \ref{subsec-conj-inv} guarantee
that it lifts to a map~$\B_{2n} \to M_{2n} (\R)$, one may wish to
use~$\theta_n^{-1}$ in order to produce a Markov function. It turns
out that this works! In fact the corresponding link invariant can be
expressed in terms of the signature and the Alexander-Conway
polynomial, so it is certainly not new. It is however striking to
think that the Weil representation should have any relationship to
links.

Let us turn to the proof. It will be concluded together with the
precise statement of Theorem~\ref{thm-weil-rep-main}. 



\subsection{An alternative construction of Meyer-additive functions} \label{subsec-alt-meyer}
Let us return to the setting of \S \ref{subsec-conj-inv}: we fix a
representation~$r \colon \B_n \to \U(V)$, and consider the
composition~$\rho \colon \B_n \to \U(\D{V})$. We keep the notation~$c$
for the two-cocycle produced.

In order to relate Thomas's construction to our own, we shall present
an alternative description of the normalized, Meyer-additive
function~$f$ corresponding to~$c$. This involves choosing a
lagrangian~$\ell$ in~$V$ in an arbitrary way, which is not only less
satisfying but will also only work when~$V$ is hyperbolic. For extreme
simplicity, we restrict to the case~$K=\R$ with trivial involution,
so~$\WH(K, \sigma ) = W(\R) \simeq \z$. Also~$\U(V) = \Sp(V)$ in this
case. This is enough for our purposes, though the reader can easily
generalize what follows.

Having chosen~$\ell$, we can consider the function~$\mu$ defined
on~$\Sp(V) \times \Sp(V)$ by 
$$ \mu(g, h) = \maslov{\ell}{g\ell}{gh\ell} \in W(\R)\, .   $$
This is again a two-cocyle on~$\Sp(V)$, different from~$c$, and we pull
it back to a two-cocyle on~$\B_n$ as before. Since~$H^2(B_n, \z) = 0$,
we deduce the existence of a map~$w : \B_n \to \WH(K, \sigma ) $ whose
coboundary is~$\mu$. 
The relationship between~$f$ and~$w$ is given by the following
Proposition, in which we write~$\beta \mapsto \langle \beta \rangle$
for the homomorphism~$B_n \to \z$ sending each~$\sigma_i$ to~$1$.

\begin{prop} \label{prop-alt-def-meyer}
There exists an integer~$k$ such that, for any~$ \beta \in \B_n$ we
have
\[ f( \beta ) = w(\beta ) + \maslov{\Gamma_\beta }{\Gamma_1} {\ell
  \oplus \ell}  + k \langle \beta \rangle \, .   \]
\end{prop}

This should be compared to Proposition 1.2 in \cite{thomas}.

\begin{proof} Put~$f'(\beta ) =  w(\beta ) + \maslov{\Gamma_\beta }{\Gamma_1} {\ell
  \oplus \ell} $. We shall prove that the function~$f'$ is
  Meyer-additive with respect to~$c$. As a result $f - f'$ is a
  homomorphism~$B_n \to \z$; however the abelianization~$B_n^{ab}$ of
  the braid group~$B_n$ is isomorphic to~$\z$ via the length map, so
  the Proposition follows.

  By definition, after a minor rearrangement of the terms, we need to
  prove that
\begin{multline*}
\maslov {\Gamma_1} {\Gamma_a} {\Gamma_{ab}} - \maslov{\ell
  \oplus\ell} {\ell \oplus a\ell} {\ell \oplus ab \ell} = \\
\maslov{\Gamma_{ab}} {\Gamma_1} {\ell\oplus \ell} + \maslov{\Gamma_1}
{\Gamma_a} {\ell\oplus\ell} + \maslov{\Gamma_1} {\Gamma_b}
{\ell\oplus\ell} \, ,
\end{multline*}
for all~$a, b \in \B_n$ (here we write~$\Gamma_a$ for~$\Gamma_{\rho
  (a)}$, and so on). Applying the unitary automorphism~$1 \times a$
(in~$\D{V}$), we see that the very last term may be replaced
by~$\maslov{\Gamma_a} {\Gamma_{ab}} {\ell\oplus a\ell}$. The situation
is summed up on the following diagram:
\begin{center}
\includegraphics[width=0.3\textwidth]{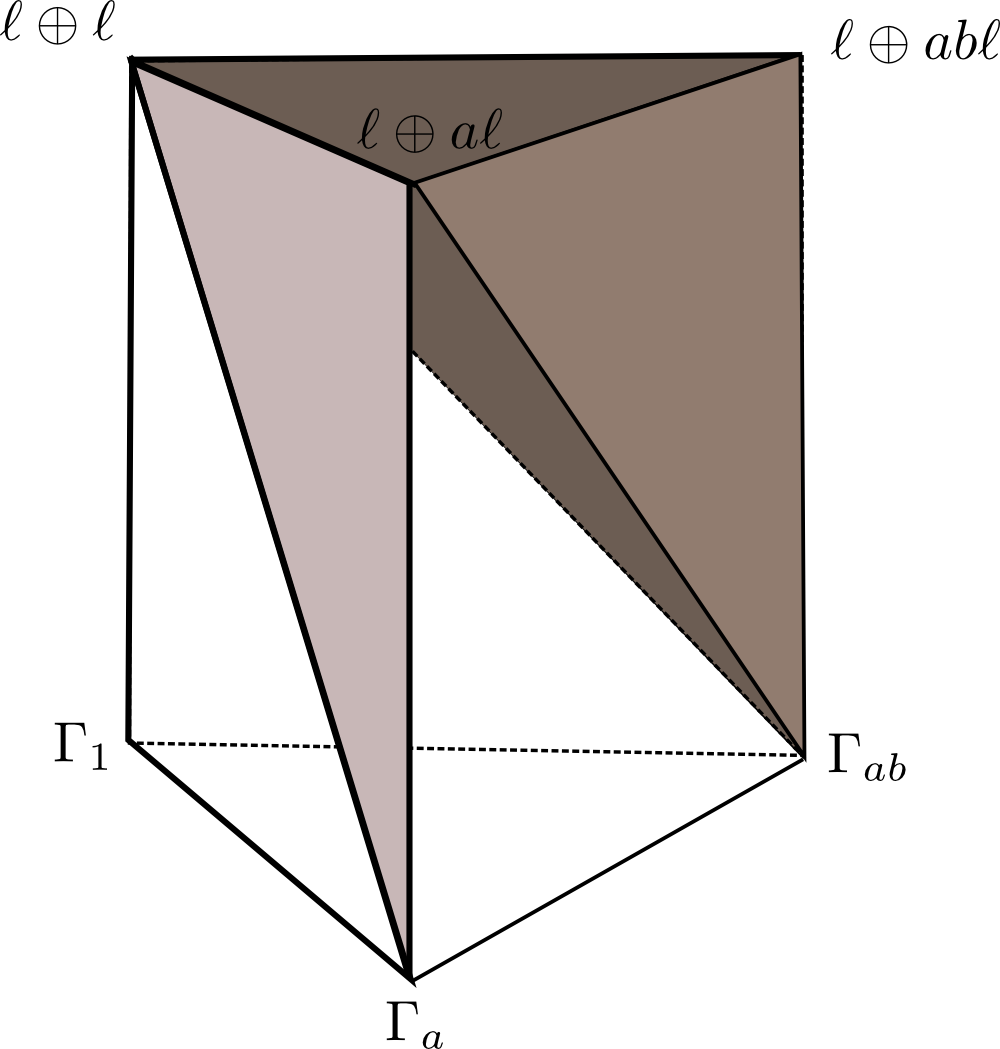}
\end{center}
On this figure, a triangulation of a triangular prism is presented
(it is meant to include~$3$-simplices). The vertices are decorated
with lagrangians, and any triangle with an orientation defines a
Maslov index unambiguously (by the dihedral symmetry property (i)). Choose
orientations consistently around the figure.

Now, the cocycle property of Maslov indices ensures that the sum of
the indices on the boundary is zero. Moreover, three of these indices
are already zero, by the lemma below: they are drawn in grey. The
result follows.
\end{proof}

We have made use of the following:

\begin{lem} Let~$g \in \U(V)$ and~$\ell, \ell'$ be lagrangians in~$V$, then
$$ \tau ( \Gamma_g ,\, \ell \oplus g\ell,\, \ell \oplus \ell') = 0 \, .   $$
\end{lem}

\begin{proof}
  Thomas's representative (as in Lemma \ref{lem-thomas}) is of
  dimension~$0$.
\end{proof}

\subsection{Thomas's model}

Since we are going to rely on the computations by Thomas
in~\cite{thomas}, we need to bridge our notation with his. 

Identify once and for all~$W(\R)$ with~$\z$ via the signature. The
``Weil character'' is the homomorphism~$\gamma \colon W(\R) \to
\C^\times$ given by~$\gamma (x)= e^{\frac{i\, \pi} {4} x}$; note that
in general there is a Weil character for each embedding~$\psi\colon \R
\to \C^\times$, and our choice corresponds to~$\psi(x) = e^{2 i \pi
  x}$ (see~\cite{lionvergne} p.\ 112). The Weil representation itself
also depends on~$\psi$, and from now we shall only consider the
representation corresponding to our choice of~$\psi$. The consistency
is important for Theorem~\ref{thm-thomas-main} below.

Let~$V$ denote~$\R^{2n}$ with the usual symplectic structure, so
that~$\Sp(V)$ refers to the group denoted~$\U(V)$ in the general
situation of \S\ref{subsec-alt-meyer}, while in
\S\ref{subsec-background} we wrote~$\Sp_{2n}(\R)$ for the same group.

Fix a lagrangian~$\ell$ for the rest of the discussion. The
cocycle~$\mu $ of \S\ref{subsec-alt-meyer} is defined, and we compose
it with~$\gamma $ to obtain a two-cocycle on~$\Sp(V)$ with values
in~$\C^\times$ (in fact, in the~$8$-th roots of unity). We are
interested in the corresponding central extension of~$\Sp(V)$ which
explicitly is the set~$M_1(V) = \Sp(V) \times \C^\times$ with multiplication 
\[ (g, t) \cdot (h, s) = (gh, ts \, \gamma(\mu (g, h)) ) \, .   \]
The group~$M_1(V)$ is ``compatible'' with our notation so far, and will
be used for explicit constructions. 

Consider now the set~$Gr(V)$ of all lagrangians in~$V$. Any pair~$(g,
t) \in M_1(V)$ gives rise to a function~$t_g \colon Gr(V) \to
\C^\times$ defined by 
\[ t_g (\ell ') = \gamma (\tau(\ell, \, g\ell, \, g \ell', \,
\ell '  ) ) \, t \, . \tag{$\dagger$}  \]
Conversely given~$g$ and the function~$t_g$, one can of course
recover~$t$ as~$t_g(\ell)$. So we can consider the group~$M_2(V)$,
canonically isomorphic to~$M_1(V)$, whose elements are all pairs~$(g,
t_g)$ where~$t_g \colon Gr(V) \to \C^\times$ is a map
satisfying~$(\dagger)$ for some complex number~$t$. A word of warning:
we borrow the notation~$t_g$ from Thomas's paper already cited, even
though it may be slightly misleading in suggesting that~$t_g$ can be
obtained from~$g$ (there is {\em no} section~$\Sp(V) \to M_2(V)$).

Finally within~$M_2(V)$, Thomas considers the subgroup~$\Mp(V)$ of
pairs~$(g, t_g)$ such that~$t_g^2 = m_g^2$, where~$m_g$ is some
function on~$Gr(V)$ whose definition is irrelevant for our
purposes. He shows that~$\Mp(V)$ is a two-fold cover of the
group~$\Sp(V)$ (see~\cite{thomas}, definition in \S5.2, Proposition
5.1 and Proposition 5.3). Since this cover is non-trivial, it must be
a model of what is universally called the metaplectic group.

We are finally in position to state Thomas's main result, computing
the values of the function~$\theta_n$ introduced in
\S\ref{subsec-background}. It is defined on~$\Mp(V)$ and Thomas uses
the notation~$\Tr \rho (g, t_g)$ (our own usage of the letter~$\rho $
is unrelated). The following is Theorem 2A in~\cite{thomas}.
\begin{thm}[Thomas] \label{thm-thomas-main}
The trace of the Weil representation is given by
\[ \theta_n(g, t_g)  = \frac{t_g(\ell) \cdot \gamma ( \maslov{\Gamma_g}
   {\Gamma_1} {\ell \oplus \ell}) } {| \det(g - 1) | ^\frac{1} {2}} \,
.  \]
\end{thm}

Let us introduce a few more groups. Let~$M_1'(V)$ be the subgroup
of~$M_1(V)$ of pairs~$(g, t)$ with~$t$ an~$8$-th root of unity, and
let~$M_2'(V)$ be the corresponding subgroup of~$M_2(V)$. Since the
function~$m_g$ takes its values also in the~$8$-th roots of unity, it
is clear that~$\Mp(V) \subset M_2'(V)$. To finish with, recall the
group~$\widetilde{\Sp(V)}$, made of pairs~$(g, n)$ with~$g\in \Sp(V)$
and~$n\in \z$ with the multiplication using the cocycle~$\mu $.

\begin{lem} \label{lem-normal}
The group~$\Mp(V)$ is normal in~$M_2'(V)$.
\end{lem}

\begin{proof}
In~\cite{lionvergne} it is proved that~$\widetilde{\Sp(V)}$ has four
connected component (for some appropriate topology which is not the
product topology on~$\Sp(V)\times \z$); more precisely in 1.7.11 in
{\em loc.\ cit.} one finds the definition of a continuous character
\[ s \colon \widetilde{\Sp(V)} \longrightarrow \z/4  \]
whose fibres are exactly the connected components. What is more,
$s(g, n)$ depends only on~$n$ modulo~$4$.

The Weil character~$\gamma $ gives a continuous and surjective map 
\[ \widetilde{\Sp(V)} \longrightarrow M_1'(V) \, ,  \]
and it is clear that~$s$ factors through~$M_1'(V)$. As a result, the
group~$M_1'(V)$ also has four connected components. They are
homeomorphic to each other, and each component is thus a two-fold cover
of~$\Sp(V)$, since~$M_1'(V)$ is an eight-fold cover.

The connected component~$G$ of the identity is a
normal subgroup of~$M_1'(V)$, so it suffices to prove that~$\Mp(V) =
G$ (identifying~$M_1'(V)$ and~$M_2'(V)$). However~$\Mp(V) \subset G$
since~$\Mp(V)$ is connected, and since these groups are both two-fold
covers of~$\Sp(V)$ one must have~$\Mp(V) = G$.
\end{proof}

\subsection{Main result}

Let the notation be as in \S \ref{subsec-background}, so that we work
over~$K=\R$ with the Burau representation at~$t=-1$, written~$r_{2n}
\colon B_{2n} \to \Sp(V)$. As in \S\ref{subsec-alt-meyer}, we pick a
one-cocycle~$w$ whose boundary is the two-cocycle~$\mu $, a
lagrangian~$\ell$ being chosen. The map~$\beta \mapsto (r_{2n}(\beta
), \gamma (w( \beta )))$ defines a lift 
\[ \widetilde{r_{2n}} \colon B_{2n} \longrightarrow M_1'(V)  \]
of~$r_{2n}$. Let us show that we can arrange for~$\widetilde{r_{2n}}$
to take its values in the metaplectic group. Indeed, one may add any
homomorphism~$B_n \to \z$ to~$w$, so we can certainly
have~$w(\sigma_1)$ taking any convenient value, and in particular we
can have~$\widetilde{r_{2n}}(\sigma_1) \in \Mp(V)$ (as before we
see~$\Mp(V)$ as a subgroup of~$M_1(V)$). Since~$\sigma_i$ is
conjugated to~$\sigma_1$ in~$B_n$, and in virtue of
Lemma~\ref{lem-normal}, this forces the image of~$\widetilde{r_{2n}}$
to lie entirely in~$\Mp(V)$.

\begin{thm} \label{thm-weil-rep-main}
There exists an integer~$k = k_n$ such that 
\[ e^{ \frac{ki \pi} {4} \langle \beta \rangle} \, \theta_n^{-1} (\widetilde{r_{2n}}(\beta )) = e^{-\frac{i\pi} {4} \Theta_\R( \hat
  \beta ) } | \det(r_{2n}(\beta ) - Id) |^{\frac{1} {2} } \, .  \]
In particular the collection~$( e^{ \frac{k_n i \pi} {4}}
\theta_n^{-1} \circ \widetilde{r_{2n}})_{n \ge 1}$ can be extended to
a Markov function.
\end{thm}

In other words, after a simple renormalization involving the braid
``exponent sum'', the trace of the Weil representation yields an
oriented link invariant. We shall not attempt to compute the value
of~$k_n$, as the link invariant does not contain ``new'' information
anyway (see comments in the proof).



\begin{proof}
Note that the right hand side of the equation to prove defines a
Markov function, since~$ | \det(r_{2n}(\beta ) - Id) |$ is the
absolute value of the Alexander-Conway polynomial at~$\sqrt{-1}$. Of
course~$\Theta_\R(\hat \beta )$ is the signature of~$\hat \beta $, as
already observed.

As for the equality itself, it now follows directly from
Theorem~\ref{thm-thomas-main} combined with
Proposition~\ref{prop-alt-def-meyer}. 
\end{proof}

\subsection{Finite fields}
The symplectic groups over finite fields also have a Weil
representation (with no need to go to a two-sheeted cover), so one may
wish to try and use it to obtain more link invariants. This was
investigated by Goldschmidt and Jones in \cite{jones}. Here is our
interpretation of their results.

Pick~$K = \F_p(t)$ for an odd prime~$p$, with involution
satisfying~$\sigma (t)=t^{-1}$, and choose~$\alpha $ so that~$\alpha
(t)=t$. Our general method produces an invariant~$\Theta_{\F_p(t)}$
with values in~$\WH(\F_p(t), \sigma )$. We would like to
specialize~$t$ to a value in some finite field~$\F_q$ ; however, there
is no field homomorphism~$\F_p(t) \to \F_q$, so this idea
cannot be pursued literally. When we prove the existence of an
invariant in the Witt ring of~$\Z[\frac{1} {2}, t, t^{-1}]$, we shall
be able to specialize~$t$ directly. For the time being, here is an
{\em ad hoc} trick, which amounts to the description given by
Goldschmidt and Jones.

Let~$u= t+t^{-1}$, so that the fixed field of~$\sigma $ is~$k=
\F_p(u)$. Recall that we have an anti-hermitian space~$V_n$, over the
field~$K$, on which~$\B_n$ acts unitarily. As explained in \S
\ref{subsec-maslov}, we may consider the ``real part'' of the
anti-hermitian form, which is a symplectic form on~$(V_n)_k$, that
is~$V_n$ viewed as a~$k$-vector space. This allows the formation of
Maslov indices, and clearly yields a link invariant, whose value on a
link~$L$ is simply the image of~$\Theta_{\F_p(t)}(L)$ under the
map~$\WH( \F_p(t), \sigma ) \to W(\F_p(u))$. What is noticeable at
this point is that the matrices giving the action of~$\B_n$
on~$(V_n)_k$, as well as the symplectic form, all have their entries
in~$\F_p[u, u^{-1}]$ and thus may be specialized to a non-zero value
of~$u$ chosen in any finite field~$\F_q$ of characteristic~$p$. If we
define our Maslov indices {\em then}, we end up with an invariant with
values in~$W(\F_q)$, which we may call~$\Theta_{\F_q}$.

Let us go back to the trace of the Weil representation
of~$\Sp_{2n}(\F_q)$. It was also computed by Thomas in
\cite{thomas}. His formula (Theorem 2B) shows that the trace of an
element of the form~$r_n( \beta )$, where~$r_n : \B_n \to
\Sp_{2n}(\F_q)$ is the representation just defined, can be expressed
in terms of~$\Theta_{\F_q}(\hat \beta )$ and the invariant of
Proposition \ref{prop-petit-markov} (details will be omitted). In
particular, it gives a link invariant.

%% file: main.bbl
\newcommand{\noopsort}[1]{} \newcommand{\printfirst}[2]{#1}
  \newcommand{\singleletter}[1]{#1} \newcommand{\switchargs}[2]{#2#1}
\begin{thebibliography}{CLM76}

\bibitem[CLM76]{cohen}
Frederick~R. Cohen, Thomas~J. Lada, and J.~Peter May.
\newblock {\em The homology of iterated loop spaces}.
\newblock Lecture Notes in Mathematics, Vol. 533. Springer-Verlag, Berlin,
  1976.

\bibitem[GG05]{ghys}
Jean-Marc Gambaudo and {\'E}tienne Ghys.
\newblock Braids and signatures.
\newblock {\em Bull. Soc. Math. France}, 133(4):541--579, 2005.

\bibitem[GJ89]{jones}
David~M. Goldschmidt and V.~F.~R. Jones.
\newblock Metaplectic link invariants.
\newblock {\em Geom. Dedicata}, 31(2):165--191, 1989.

\bibitem[KT08]{kassel}
Christian Kassel and Vladimir Turaev.
\newblock {\em Braid groups}, volume 247 of {\em Graduate Texts in
  Mathematics}.
\newblock Springer, New York, 2008.
\newblock With the graphical assistance of Olivier Dodane.

\bibitem[Lic97]{lick}
W.~B.~Raymond Lickorish.
\newblock {\em An introduction to knot theory}, volume 175 of {\em Graduate
  Texts in Mathematics}.
\newblock Springer-Verlag, New York, 1997.

\bibitem[LV80]{lionvergne}
G{\'e}rard Lion and Mich{\`e}le Vergne.
\newblock {\em The {W}eil representation, {M}aslov index and theta series},
  volume~6 of {\em Progress in Mathematics}.
\newblock Birkh\"auser Boston, Mass., 1980.

\bibitem[MH73]{milnor}
John Milnor and Dale Husemoller.
\newblock {\em Symmetric bilinear forms}.
\newblock Springer-Verlag, New York, 1973.
\newblock Ergebnisse der Mathematik und ihrer Grenzgebiete, Band 73.

\bibitem[Rei83]{reide}
K.~Reidemeister.
\newblock {\em Knot theory}.
\newblock BCS Associates, Moscow, Idaho, 1983.
\newblock Translated from the German by Leo F. Boron, Charles O. Christenson
  and Bryan A. Smith.

\bibitem[Tho08]{thomas}
Teruji Thomas.
\newblock The character of the {W}eil representation.
\newblock {\em J. Lond. Math. Soc. (2)}, 77(1):221--239, 2008.

\end{thebibliography}
